\documentclass[a4paper]{amsproc}

\theoremstyle{definition}

\theoremstyle{remark}

\numberwithin{equation}{section}
\usepackage{graphicx}
\usepackage{latexsym}
\usepackage{amsmath}
\usepackage{amssymb}
\usepackage{amsfonts}
\usepackage{verbatim}
\usepackage{mathrsfs}
\usepackage[modulo]{lineno} 
\usepackage[latin1]{inputenc}
\usepackage{color}
\usepackage[colorlinks,citecolor=blue,urlcolor=blue]{hyperref}

\newtheorem{tm}{Theorem}[section]
\newtheorem{rk}{Remark}[section]
\newtheorem{ap}{Assumption}[section]

\newtheorem{prop}{Proposition}[section]
\newtheorem{lm}{Lemma}[section]

\newcommand{\ee}{\mathbb E}
\newcommand{\ff}{\mathbb F}
\newcommand{\pp}{\mathbb P}
\newcommand{\nn}{\mathbb N}
\newcommand{\rr}{\mathbb R}

\newcommand{\zz}{\mathbb Z}

\newcommand{\LL}{\mathcal L}

\newcommand{\PP}{\mathcal P}

\newcommand{\TT}{\mathcal T}

\newcommand{\OOO}{\mathscr O}
\newcommand{\FFF}{\mathscr F}

\newcommand{\<}{\langle}
\renewcommand{\>}{\rangle}
\allowdisplaybreaks \allowdisplaybreaks[4]

\begin{document}
 
\title[$L^p$-Convergence Rate of BE Schemes for Monotone SDEs]
{$L^p$-Convergence Rate of Backward Euler Schemes for Monotone SDEs}

\author{Zhihui LIU}
\address{Department of Mathematics, Southern University of Science and Technology, 1088 Xueyuan Avenue, Shenzhen 518055, P.R. China}
\curraddr{}
\email{liuzh3@sustech.edu.cn}



\date{\today}

\dedicatory{}

\begin{abstract}
We give a unified method to derive the strong convergence rate of the backward Euler scheme for monotone SDEs in $L^p(\Omega)$-norm, with general $p \ge 4$.
The results are applied to the backward Euler scheme of SODEs with polynomial growth coefficients.
We also generalize the argument to the Galerkin-based backward Euler scheme of SPDEs with polynomial growth coefficients driven by multiplicative trace-class noise.
\end{abstract}

\maketitle



\section{Introduction}

There is a general theory on strong error estimations for SODEs and SPDEs with Lipschitz coefficients; see, e.g., the monograph \cite{KP92, Kru14} and the references therein. 
For SODEs with one-sided Lipschitz continuous coefficients which grow super-linearly, the classical Euler--Maruyama (EM) scheme is known to diverge \cite{HJK11(PRSLSA)}. 
In particular, the authors in \cite{HJK11(PRSLSA)} obtained the divergence of $p$-moments for the EM scheme for all $p \in [1, \infty)$.
Recently, this divergence result is generalized to second-order parabolic SPDEs driven by white noise \cite{BHJKLS19}. 
Based on this type of divergence result, one has to use explicit or implicit strategies to design convergent numerical schemes for SDEs with polynomial growth coefficients.
 
For SODEs with polynomial growth coefficients, there are kinds of literature using explicit strategy including adaptive time step size technique \cite{LMS07(IMA), KL18(IMA), FG20(AAP)} and tamed or truncated argument \cite{CJM16(SIFM), HJK12(AAP), HJ15(MAMS), Sab16(AAP)} based on renormalized-increments to the scheme.
Besides, implicit strategy including backward Euler (BE) scheme or its split-step version \cite{HMS02(SINUM), AK17(BIT)} and stochastic $\theta$ scheme \cite{WWD20(BIT)} is also investigated. 
These explicit or implicit strategies are generalized to monotone SPDEs with polynomial growth coefficients; see, e.g., \cite{BJ19(SPA), BGJK17, JP20(IMA), Wan20(SPA)} using tamed or truncated argument, \cite{KLL15(JAP), KLL18(MN), FLZ17(SINUM), LQ20(IMA), LQ20(SPDE), QW19(JSC), QW20(SINUM)} using BE scheme or its split-step version; see also \cite{BCH19(IMA)} using time splitting scheme.
We also refer to \cite{Dor12(SINUM), CHL17(JDE), CHLZ19(JDE), HJ20(AOP)} where the authors used exponential integrability of both the exact and numerical solutions to SPDEs with non-Lipchitz coefficients, including the 2D stochastic Navier--Stokes equations, 1D stochastic nonlinear Schr\"odinger equation, 1D Cahn--Hilliard--Cook equation, and 1D stochastic Burgers equation.


When using explicit strategy, some of the previous results got the convergence rate in $L^p(\Omega)$-norm with $p \ge 2$.
For implicit strategy, most of the previous results derived the mean-square convergence rate, i.e., the convergence rate in $L^p(\Omega)$-norm with $p=2$. 
One of the main reasons why it is not a straight forward consequence of $L^p(\Omega)$-convergence rate for general $p$ is that it is not easy to derive the uniform boundedness of the $p$-moment of the implicit scheme in the general multiplicative noise case; see \cite{QW19(JSC)} for $L^p(\Omega)$-convergence rate for monotone SPDEs driven by additive noise. 

The main aim of the present paper is to establish the strong convergence rate of the BE scheme \eqref{be} for monotone SODEs and Galerkin-bassed BE schemes \eqref{be1} and \eqref{be2} for monotone SPDEs in $L^p(\Omega)$-norm with general $p \ge 4$.
We point out that our results are also valid for $p=2$.
Our first step is to derive the uniform boundedness of the $p$-moment of the BE scheme and related continuous-time interpolation process (see Proposition \ref{sta}) in Section \ref{sec2}.
Then we give the $L^p$-error representation between the exact solution and the BE scheme and interpolation process (see Theorems \ref{main1} and \ref{main2}), respectively. 
With the polynomial growth condition, we obtain the desired $L^p$-convergence rate which is exactly the same as in the Lipschitz case (see Theorem \ref{main3}). 
The proofs of these $L^p$-error representation and convergence rate are given in Section \ref{sec3}. 
These arguments and results are extended in the last section to monotone SPDEs.

\section{Preliminaries and Main Results}
\label{sec2}

Let $T \in (0, \infty)$ be fixed.
Consider the SODE 
\begin{align} \label{sde}
\begin{split}  
&{\rm d} X_t
=b(X_t) {\rm d}t + \sigma(X_t) {\rm d}W_t,  
\quad t \in (0, T], \\
&X(0)=X_0, 
\end{split}
\end{align}  
driven by an $\rr^m$-valued Wiener process $W$ on a filtered probability space $(\Omega, \FFF, \ff:=\{\FFF_t\}_{t \ge 0}, \pp)$, where $b: \rr^n \rightarrow \rr^n$ and $\sigma:\rr^n \rightarrow \rr^{n \times m}$ are continuous functions satisfying certain monotone conditions, and $X_0$ is an $\rr^n$-valued $\FFF_0$-measurable random variable.
Our main condition on the coefficients $b$ and $\sigma$ in Eq. \eqref{sde} is the following monotone  assumption.

\begin{ap} \label{ap-mon}
There exist $L_b \in \rr$ and $L_\sigma \ge 0$ such that for all $x,y\in \rr^n$, 
\begin{align} 
\<b(x)-b(y), x-y\> & \le L_b |x-y|^2, \label{A1} \tag{A1} \\ 
|\sigma(x)-\sigma(y)|_{HS} & \le L_{\sigma} |x-y|.  \label{B1} \tag{B1}
\end{align} 
\end{ap}

In the above and throughout Sections \ref{sec2} and \ref{sec3}, $|\cdot|$ denotes the Euclidean norm on $\rr^n$ and
$|\cdot|_{HS}$ denotes the Hilbert--Schmidt norm on $\rr^{n \times m}$: $|\sigma|_{HS}=(\sum_{i=1}^n \sum_{i=1}^m \sigma_{ij}^2)^{1/2}$.

\begin{rk} 
The conditions \eqref{A1} and \eqref{B1} yield the following coercivity condition for any $\epsilon>0$ and $x,y\in \rr^n$: 
\begin{align} 
& \<b(x), x\> \le (L_b+\epsilon)|x|^2+\frac1{4 \epsilon}|b(0)|^2, \label{coe} \\
& |\sigma(x)|_{HS} \le L_\sigma |x|+|\sigma(0)|_{HS}, \label{gro-g}
\end{align} 
and thus the following monotone and coercivity conditions with $L:=2 L_b+ L_\sigma^2$ and some constant $\alpha, \beta \ge 0$ hold for all $x,y\in \rr^n$: 
\begin{align*} 
& 2 \<b(x)-b(y), x-y\> + |\sigma(x)-\sigma(y)|_{HS}^2
\le L |x-y|^2, \\ 
& 2 \<b(x), x\> + |\sigma(x)|_{HS}^2 \le \alpha+\beta |x|^2.
\end{align*} 
Therefore, Eq. \eqref{sde} possesses a unique $\ff$-adapted solution $\{X_t\}_{t \in [0, T]}$ with continuous sample paths (see \cite[Theorem 3.14]{KR79} or \cite[Theorem 3.1.1]{LR15}). 
Moreover, if $\ee |X_0|^p<\infty$ for some $p \ge 2$, then there exists $C$ such that (see \cite[Lemma 3.2]{HMS02(SINUM)})
\begin{align} \label{sta-sde}
\ee \sup_{t \in [0, T]} |X_t|^p  \le C (1+ \ee |X_0|^p).
\end{align} 
\end{rk}

\subsection{Stability of Backward Euler Scheme and Related Interpolation Process}

Let $T<M \in \nn_+$ and $\tau:=T/M \in (0, 1)$.
Denote $\zz_M:=\{0,1,\cdots,M\}$ and $\zz_M^*:=\{1,\cdots,M\}$.
Under (A1) and $L_b \tau<1$, the following backward Euler (BE) scheme is (almost surely) solvable:
\begin{align}\label{be}
Y^M_{k+1}=Y^M_k + b(Y^M_{k+1}) \tau
+ \sigma(Y^M_k) \delta_k W, \quad k \in \zz_{M-1}; \quad Y^M_0=X_0.
\end{align}  

Introduce the continuous-time interpolation process 
\begin{align} \label{int}
Y^M_t=Y^M_k + b(Y^M_{k+1}) (t- t_k) + \sigma(Y^M_k) (W_t-W_{t_k}), 
\end{align} 
with $t \in [t_k, t_{k+1})$, $k \in \zz_{M-1}$, and $Y^M_T:=Y^M_M$.  
Then $\{Y^M_t:~t \in [0, T]\}$ is an It\^o process, which can be rewritten as  
\begin{align}  \label{int-ito}
Y^M_t=X_0 + \int_0^t b(Y^M_{\eta^+_M(s)}) {\rm d} s
+ \int_0^t \sigma(Y^M_{\eta_M(s)}) {\rm d}W_s, 
\end{align} 
where $\eta^+_M(t):=t_{k+1}$ and $\eta_M(t):=t_k$ when $t \in [t_k, t_{k+1})$, $k \in \zz_{M-1}$.

The main result in this part is the following moments' stability of the above BE scheme \eqref{be} and interpolation process \eqref{int}.
In what follows $C$ denotes a generic constant which would be different in each appearance.

\begin{prop} \label{sta} 
Let \eqref{A1} and \eqref{B1} hold for some $L_b \in \rr$ and $L_\sigma \ge 0$.
Assume that $\ee |X_0|^p<\infty$ for some $p \ge 4$, then for any $M \in \nn_+$ such that $2 L_b \tau<1$, there exists $C$ such that
\begin{align} \label{sta-yk}
\ee \sup_{k \in \zz_M} |Y^M_k|^p
\le C (1+ \ee |X_0|^p).
\end{align}      
Assume furthermore that $\sup_{k \in \zz_M} \ee |b(Y^M_k)|^p<\infty$, then 
\begin{align} \label{sta-int} 
\ee \sup_{t \in [0, T]} |Y^M_t|^p 
\le C  (1+ \ee |X_0|^p+ \sup_{k \in \zz_M} \ee |b(Y^M_k)|^p).
\end{align}   
\end{prop} 
 
\begin{proof} 
For simplicity, set $b_{k+1}=b(Y^M_{k+1})$ and $\sigma_k=\sigma(Y^M_k)$.
Testing \eqref{be} by $Y^M_{k+1}$ and using  
\begin{align} \label{ab}
\<\alpha-\beta, \alpha\>
=\frac12 (|\alpha|^2-|\beta|^2)+\frac12 |\alpha-\beta|^2
\end{align} 
for $\alpha, \beta\in \rr^n$,
we have  
\begin{align*} 
& \frac12 ( |Y^M_{k+1}|^2 - |Y^M_k|^2) 
+  \frac12 |Y^M_{k+1}-Y^M_k|^2  \\
&= \<Y^M_{k+1}, b_{k+1}\> \tau
+\<Y^M_{k+1}-Y^M_k, \sigma_k \delta_k W\>
+\<Y^M_k, \sigma_k \delta_k W\>.
\end{align*}  
By the conditions \eqref{coe}-\eqref{gro-g} and Cauchy--Schwarz inequality, there exists $C=C(\epsilon, |b(0)|,|\sigma(0)|_{HS})$ such that 
\begin{align*} 
& \frac12 ( |Y^M_{k+1}|^2 - |Y^M_k|^2) 
+  \frac12 |Y^M_{k+1}-Y^M_k|^2  \\ 
&\le C (\tau+|\delta_k W|^2) + (L_b+\epsilon) |Y^M_{k+1}|^2 \tau +  \frac12 |Y^M_{k+1}-Y^M_k|^2 \\
& \quad +L_\sigma^2 |Y^M_k|^2 |\delta_k W|^2
+\<Y^M_k, \sigma_k \delta_k W\> 
\end{align*}  
It follows that
\begin{align*}  
& (1-2 (L_b+\epsilon)  \tau) |Y^M_{k+1}|^2
\le C (\tau+|\delta_k W|^2) 
+ (1+2 L_\sigma^2 |\delta_k W|^2) |Y^M_k|^2
+\<Y^M_k, \sigma_k \delta_k W\>,
\end{align*}  
and thus for $k \in \zz_M^*$,
\begin{align*}   
& (1-2 (L_b+\epsilon) \tau) |Y^M_k|^2   \\
& \le C +C \sum_{i=0}^{k-1} |\delta_i W|^2 
+ \sum_{i=0}^{k-1} (2 (L_b+\epsilon) \tau+L_\sigma^2 |\delta_k W|^2) |Y^M_i|^2 +\sum_{i=0}^{k-1} \<Y^M_i, \sigma_i \delta_i W\>.
\end{align*}  
Since $2 L_b \tau<1$, one can choose $\epsilon>0$ such that $1-2(L_b+\epsilon) \tau \ge C_0$ for some $C_0>0$.
Then
\begin{align*}    
& \ee \sup_{k \in \zz_M^*} |Y^M_k|^p
\le C + C \ee \Big(\sum_{i=0}^{M-1} |\delta_i W|^2 \Big)^\frac p2
+ C \ee \Big( \sum_{i=0}^{M-1}(\tau+|\delta_i W|^2)  |Y^M_i|^2\Big)^\frac p2 \\
& \qquad  \qquad  \qquad \quad +C \ee \sup_{k \in \zz_M^*} \Big|\sum_{i=0}^{k-1} \<Y^M_i, \sigma_i \delta_i W\>\Big|^\frac p2.
\end{align*} 
Using the elementary inequality 
\begin{align} \label{sum}
(\sum_{i=0}^{M-1} a_i)^\gamma 
\le M^{\gamma-1} \sum_{i=0}^{M-1} a_i^\gamma,
\quad \forall~\gamma \ge 1, ~ a_i \ge 0,
\end{align} 
and the independence between $\delta_i W$ and $Y^M_i$,  we derive   
\begin{align*} 
& \ee \Big(\sum_{i=0}^{M-1} |\delta_i W|^2 \Big)^\frac p2
\le M^{\frac p2-1} \sum_{i=0}^{M-1} \ee |\delta_i W|^p
\le C \\  
& \ee \Big( \sum_{i=0}^{M-1}(\tau+|\delta_i W|^2)  |Y^M_i|^2\Big)^\frac p2
\le \Big( \sum_{i=0}^{M-1} \|(\tau+|\delta_i W|^2)  |Y^M_i|^2\|_{L^\frac p2 (\Omega)} \Big)^\frac p2 \\
& \le C \tau^\frac p2 \Big( \sum_{i=0}^{M-1} \| Y^M_i\|^2_{L^p (\Omega)} \Big)^\frac p2 
\le C \tau \sum_{i=0}^{M-1} \ee | Y^M_i|^p.
\end{align*} 
The same argument, in combination with discrete Burkholder inequality, yields 
\begin{align} \label{dis-BDG}  
& \ee \sup_{k \in \zz_M^*} \Big|\sum_{i=0}^{k-1} \<Y^M_i, \sigma_i \delta_i W\>\Big|^\frac p2
\le C \ee \Big(\sum_{i=0}^{M-1} \<Y^M_i, \sigma_i \delta_i W\>^2 \Big)^\frac p4  
 \le C + C \tau \sum_{i=0}^{M-1} \ee | Y^M_i|^p,
\end{align}  
for any $p \ge 4$.
Combining the above three inequalities implies
\begin{align*}  
\ee \sup_{k \in \zz_M} |Y^M_k|^p 
\le C + \ee |X_0|^p+ C \tau \sum_{i=0}^{M-1} \ee | Y^M_i|^p. 
\end{align*} 
We conclude \eqref{sta-yk} by the above inequality and Gr\"onwall lemma.
 
To show \eqref{sta-int}, we apply It\^o formula to the interpolation process $Y^M_t$, the conditions \eqref{coe}-\eqref{gro-g} and \eqref{A2}, and Burkholder--Davis--Gundy inequality to get  
\begin{align*}
\ee \sup_{t \in [0, T]} |Y^M_t|^p 
& \le \ee |X_0|^p + C \int_0^T \ee |b(Y^M_{\eta^+_M})|^p+ \ee |\sigma(Y^M_{\eta_M})|^p {\rm d}s \\
& \le \ee |X_0|^p + C \big(1 + \sup_{k \in \zz_M} \ee |b(Y^M_k)|^p+ \sup_{k \in \zz_M} \ee |Y^M_k|^p \big).
\end{align*}
Then we get \eqref{sta-int} by the above estimate and the estimate \eqref{sta-yk}. 
\end{proof}

\begin{rk} \label{rk-2}
When $p=2$, as $\ee \sum_{i=0}^{k-1} \<Y^M_i, \sigma_i \delta_i W\>=0$, one do not need to control the stochastic term in Proposition \ref{sta}, so that we can use the same argument in Proposition \ref{sta} to get the following type of second moment's estimation: 
\begin{align*}
\sup_{k \in \zz_M} \ee |Y^M_k|^2 \le C (1+ \ee |X_0|^2),
\end{align*}   
provided \eqref{A1} and \eqref{B1} hold and $\ee |X_0|^2<\infty$.   
\end{rk}

\begin{rk}  
The $p$-moment's estimation with $p \in (2,4)$ is unknown. 
In our argument, to apply \eqref{sum} the restriction $p \ge 4$ is needed in the second inequality of \eqref{dis-BDG}.  
This is the main reason why our results (e.g., Theorems \ref{main1}, \ref{main3}, and \ref{tm-err-spde}) hold true only for $p=2$ and $p \ge 4$.
\end{rk}

\subsection{Main Results}

Now we are in the position to present our main results. 
Our first result is the following representation of the $L^p$-error estimate between the exact solution $\{X_{t_k}:~ k \in \zz_M\}$ of Eq. \eqref{sde}and the backward Euler scheme $\{Y_k^M:~ k \in \zz_M\}$ defined in \eqref{be}.

\begin{tm} \label{main1}
Let Assumption \ref{ap-mon} hold for some $L_b \in \rr$ and $L_\sigma \ge 0$ and 
$\sup_{t \in [0, T]} \ee |b(X_t)|^p+\sup_{t \in [0, T]} \ee |X_t|^p<\infty$ for some $p \ge 4$.
Then for any $M \in \nn_+$ such that $2 L_b \tau<1$, there exists $C$ such that  
\begin{align}  \label{main10}
\ee \sup_{k \in \zz_M} |X_{t_k}-Y_k^M|^p 
\le C \sum_{i=0}^{M-1} \int_{t_i}^{t_{i+1}} \ee |b(X_r)-b(X_{t_{i+1}})|^p + \ee |X_r-X_{t_i}|^p {\rm d}r.
\end{align} 
\end{tm}

The second result is the following representation of the $L^p$-error estimate between the exact solution $\{X_t:~ t \in [0, T]\}$ of Eq. \eqref{sde}and the interpolation process $\{Y_t^M:~ t \in [0, T]\}$ defined in \eqref{int}.

\begin{tm} \label{main2} 
Under the same conditions as in Theorem \ref{main1}, 
for any $M \in \nn_+$ such that $2 L_b \tau<1$, there exists s constant $C$ such that  
\begin{align}  \label{main20}
\ee \sup_{t \in [0, T]} |X_t-Y_t^M|^p
& \le C \int_0^T \ee |b(Y^M)-b(Y^M_{\eta^+_M})|^p
+ \ee |Y^M-Y^M_{\eta_M}|^p {\rm d} r. 
\end{align} 
\end{tm} 

Here and after, we hide the corresponding integrated variable for simplicity.
With the following polynomial growth conditions on the drift function $b$, we have the same $L^p$-convergence rates between the BE scheme \eqref{be} or the interpolation process \eqref{int} and the exact solution of Eq. \eqref{sde} as the Lipschitz case.

\begin{ap} \label{ap-b} 
There exist $\widetilde{L_b} \ge 0$ and $q \ge 1$ such that for $x,y\in \rr^n$, 
\begin{align} \label{A2}\tag{A2} 
|b(x)-b(y)| \le \widetilde{L_b} (1+|x|^{q-1}+|y|^{q-1}) |x-y|.  
\end{align}    
\end{ap}

\begin{rk} 
A motivating example of $b$ such that \eqref{A1} and \eqref{A2} hold true is a polynomial of odd order $q$ with a negative leading coefficient, perturbed with a Lipschitz continuous function.
\end{rk}

\begin{tm}  \label{main3}
Let Assumptions \ref{ap-mon} and \ref{ap-b} hold for some $L_b \in \rr$, $\widetilde{L_b}, L_\sigma \ge 0$, and $q \ge 1$.
Assume that $\ee |X_0|^{p (2q-1)}<\infty$ for some $p \ge 4$, then for any $M \in \nn_+$ such that $2 L_b \tau<1$, there exists $C$ such that   
\begin{align}  \label{err-sde}
\ee \sup_{k \in \zz_M} |X_{t_k}-Y_k^M|^p 
\le C (1+\ee|X_0|^{p (2q-1)}) \tau^\frac p2.
\end{align}  
Assume that $\ee |X_0|^{pq (2q-1)}<\infty$ for some $p \ge 4$, then for any $M \in \nn_+$ such that $2 L_b \tau<1$, there exists $C$ such that   
\begin{align}  \label{err-sde+}
\ee \sup_{t \in [0, T]} |X_t-Y_t^M|^p 
\le C (1+\ee|X_0|^{pq (2q-1)}) \tau^\frac p2. 
\end{align}  
\end{tm}

\begin{rk}
As pointed out in Remark \ref{rk-2}, the error representations \eqref{main10}-\eqref{main20} and the error estimates \eqref{err-sde}-\eqref{err-sde+} are also valid when $p=2$ under the $l^\infty(\zz_M; L^2(\Omega))$-norm and the $L^\infty(0, T; L^2(\Omega))$-norm, respectively.
The estimate \eqref{err-sde} with $p=2$ under the $l^\infty(\zz_M; L^2(\Omega))$-norm had been shown firstly in \cite{HMS02(SINUM)}; see also \cite{WWD20(BIT)} where the authors used another type of representations for $\sup_{k \in \zz_M} \ee |X_{t_k}-Y_k^M|^2$.
\end{rk}

\subsection{Possible Generalization to SODEs with Exponential Growth Coefficients}

The purpose of this part is to give possible generalization to SODEs with exponential growth drifts for some $\widetilde{L_b}, \widehat{L_b} \ge 0$:
\begin{align} \label{exp} 
|b(x)-b(y)| \le \widetilde{L_b} (1+e^{\widehat{L_b} |x|}+e^{\widehat{L_b}|y|}) |x-y|,  \quad x, y \in \rr^n.
\end{align}   
A motivating example is $b(x)=e^{-x}$, $x \in \rr$, perturbed by an odd polynomial with a negative leading coefficient and a Lipschitz function. 

Using the representation \eqref{main10} and the condition \eqref{exp}, we have
\begin{align*}  
& \ee \sup_{k \in \zz_M} |X_{t_k}-Y_k^M|^p \\ 
& \le C (1+ \sup_{t \in [0, T]} \ee e^{2p \widehat{L_b}  |X_r|})^\frac12 
\sum_{i=0}^{M-1} \int_{t_i}^{t_{i+1}}  (\ee |X_r-X_{t_i}|^{2p}+\ee |X_r-X_{t_{i+1}}|^{2p})^\frac12 {\rm d} r \\   
& \le C (1+ \sup_{t \in [0, T]} \ee e^{2p \widehat{L_b} |X_r|}) \tau^\frac p2.
\end{align*} 
It is not clear whether one can apply the exponential integrability criteria in \cite{CHL17(JDE), CHLZ19(JDE)} and \cite{HJ20(AOP)} to derive the boundedness of  
$\sup_{t \in [0, T]} \ee e^{2p \widehat{L_b} |X_r|}$ depending on certain exponential moments of $X_0$.
Once this boundedness is established, we obtain the $L^p$-error estimate:
\begin{align*}  
\ee \sup_{k \in \zz_M} |X_{t_k}-Y_k^M|^p 
\le C \tau^\frac p2.
\end{align*}

\section{$L^p$-Convergence Rate of BE Schemes of Monotone SODEs}
\label{sec3}

\subsection{Proof of Theorem \ref{main1}}

For $k \in \zz_M$, denote $e_k:=X_{t_k}-Y_k^M$. 
It is clear that for all $k \in \zz_{M-1}$,
\begin{align*}
X_{t_{k+1}} =X_{t_k}+\int_{t_k}^{t_{k+1}} b(X_r) {\rm d}r + \int_{t_k}^{t_{k+1}} \sigma(X_r) {\rm d}W_r. 
\end{align*}  
The above equality and \eqref{be} yield that 
\begin{align*}
e_{k+1}-e_k & =\int_{t_k}^{t_{k+1}} b(X_r)-b(X_{t_{k+1}}) {\rm d}r + \int_{t_k}^{t_{k+1}} \sigma(X_r)-\sigma(X_{t_k}) {\rm d}W_r \\
& \quad +(b(X_{t_{k+1}})-b(Y_{k+1}^M)) \tau 
+(\sigma(X_{t_k})-\sigma(Y_k^M)) \delta_k W. 
\end{align*} 

Testing $e_{k+1}$ and using the identity \eqref{ab},
we have   
\begin{align} \label{eq-err}
& \frac12 ( |e_{k+1}|^2 - |e_k|^2) 
+  \frac12 |e_{k+1}-e_k|^2  \nonumber \\
&= \<e_{k+1}, \int_{t_k}^{t_{k+1}} b(X_r)-b(X_{t_{k+1}}) {\rm d}r\> 
+ \<e_{k+1}, b(X_{t_{k+1}})-b(Y_{k+1}^M)\> \tau  \nonumber \\
& \quad + \<e_{k+1}, \int_{t_k}^{t_{k+1}} \sigma(X_r)-\sigma(X_{t_k}) {\rm d}W_r\>
+ \<e_{k+1}, (\sigma(X_{t_k})-\sigma(Y_k^M)) \delta_k W\>.
\end{align}  
By Cauchy--Schwarz inequality and the condition \eqref{A1}, we get 
\begin{align*}
& \<e_{k+1}, \int_{t_k}^{t_{k+1}} b(X_r)-b(X_{t_{k+1}}) {\rm d}r\> 
+ \<e_{k+1}, b(X_{t_{k+1}})-b(Y_{k+1}^M)\> \tau \\
& \le (L_b+\epsilon) \tau|e_{k+1}|^2 + \frac1{4 \epsilon} \int_{t_k}^{t_{k+1}} |b(X_r)-b(X_{t_{k+1}})|^2 {\rm d}r, \quad \forall ~ \epsilon>0, \\
& \<e_{k+1}, \int_{t_k}^{t_{k+1}} \sigma(X_r)-\sigma(X_{t_k}) {\rm d}W_r\>
+ \<e_{k+1}, (\sigma(X_{t_k})-\sigma(Y_k^M)) \delta_k W\> \\
& = \<e_{k+1}-e_k, \int_{t_k}^{t_{k+1}} \sigma(X_r)-\sigma(X_{t_k}) {\rm d}W_r\> +\<e_k, \int_{t_k}^{t_{k+1}} \sigma(X_r)-\sigma(X_{t_k}) {\rm d}W_r\> \\
& \quad + \<e_{k+1}-e_k, (\sigma(X_{t_k})-\sigma(Y_k^M)) \delta_k W\>
+ \<e_k, (\sigma(X_{t_k})-\sigma(Y_k^M)) \delta_k W\> \\
& \le \frac12 |e_{k+1}-e_k|^2 + \Big|\int_{t_k}^{t_{k+1}} \sigma(X_r)-\sigma(X_{t_k}) {\rm d}W_r \Big|^2 
+ L_\sigma^2 |e_k|^2 |\delta_k W|^2 \\
& \quad + \<e_k, \int_{t_k}^{t_{k+1}} \sigma(X_r)-\sigma(X_{t_k}) {\rm d}W_r\>
+\<e_k, (\sigma(X_{t_k})-\sigma(Y_k^M)) \delta_k W\>.
\end{align*} 
Substituting the above inequalities into \eqref{eq-err}, we obtain
\begin{align*} 
& (1-2(L_b+\epsilon) \tau) |e_{k+1}|^2 - (1+2 L_\sigma^2 |\delta_k W|^2)|e_k|^2   \\
& \le \frac1{2 \epsilon} \int_{t_k}^{t_{k+1}} |b(X_r)-b(X_{t_{k+1}})|^2 {\rm d}r   + 2 \Big|\int_{t_k}^{t_{k+1}} \sigma(X_r)-\sigma(X_{t_k}) {\rm d}W_r \Big|^2  + 2 R_k,
\end{align*}  
where $R_k:=\<e_k, \int_{t_k}^{t_{k+1}} \sigma(X_r)-\sigma(X_{t_k}) {\rm d}W_r\>+\<e_k, (\sigma(X_{t_k})-\sigma(Y_k^M)) \delta_k W\>$.
Then for $k=1,\cdots,M$,
\begin{align} \label{ek}
& (1-2(L_b+\epsilon) \tau) |e_k|^2 
\le 2 \sum_{i=0}^{k-1} ((L_b+\epsilon)\tau +L_\sigma^2 |\delta_i W|^2) |e_i|^2
+ 2 \sum_{i=0}^{k-1} R_i \nonumber \\
& + \frac1{2 \epsilon} \sum_{i=0}^{k-1} \int_{t_i}^{t_{i+1}} |b(X_r)-b(X_{t_{i+1}})|^2 {\rm d}r
+ 2 \sum_{i=0}^{k-1} \Big|\int_{t_k}^{t_{k+1}} \sigma(X_r)-\sigma(X_{t_k}) {\rm d}W_r \Big|^2. 
\end{align}
 
For any $M \in \nn_+$ such that $2 L_b \tau<1$, one can choose $\epsilon>0$ such that $1-2(L_b+\epsilon) \tau \ge C_0$ for some $C_0>0$.
For any $p \ge 2$, the inequality \eqref{ek} yields the existence of $C=C(\epsilon, p)$ such that  
\begin{align*}
& \ee \sup_{k \in \zz_M^*} |e_k|^p  \\
& \le C \ee \Big(\sum_{i=0}^{M-1} (\tau + |\delta_i W|^2) |e_i|^2\Big)^\frac p2
+ C \ee \Big(\sum_{i=0}^{M-1} \int_{t_i}^{t_{i+1}} |b(X_r)-b(X_{t_{i+1}})|^2 {\rm d}r\Big)^\frac p2   \\
& \quad + C \ee \Big(\sum_{i=0}^{M-1} \Big|\int_{t_k}^{t_{k+1}} \sigma(X_r)-\sigma(X_{t_k}) {\rm d}W_r \Big|^2\Big)^\frac p2
+ C \ee \sup_{k \in \zz_M^*} \Big|\sum_{i=0}^{k-1} R_i\Big|^\frac p2. 
\end{align*}
Using the inequality \eqref{sum}, the independence between $\delta_i W$ and $e_i$, H\"older and Burkholder--Davis--Gundy inequalities, and the condition \eqref{A1}, we derive
\begin{align} \label{err-ek}
\ee \sup_{k \in \zz_M^*} |e_k|^p  
& \le C \tau \sum_{i=0}^{M-1} \ee |e_i|^p
+ C \sum_{i=0}^{M-1} \int_{t_i}^{t_{i+1}} \ee|b(X_r)-b(X_{t_{i+1}})|^p {\rm d}r \nonumber \\
& \quad + C \sum_{i=0}^{M-1} \int_{t_k}^{t_{k+1}} \ee|X_r-X_{t_k}|^p {\rm d}r
+ C \ee \sup_{k \in \zz_M^*} \Big|\sum_{i=0}^{k-1} R_i\Big|^\frac p2. 
\end{align}
Finally, since $e_k$ is independent of  $\int_{t_k}^{t_{k+1}} \sigma(X_r)-\sigma(X_{t_k}) {\rm d}W_r$ and $(\sigma(X_{t_k})-\sigma(Y_k^M)) \delta_k W$, we have by using discrete Burkholder inequality as in \eqref{dis-BDG} with $p \ge 4$ that 
\begin{align*}
\ee \sup_{k \in \zz_M^*} \Big|\sum_{i=0}^{k-1} R_i \Big|^\frac p2 
\le C \tau \sum_{i=0}^{M-1} \ee |e_i|^p
+ C \sum_{i=0}^{M-1} \int_{t_i}^{t_{i+1}} \ee |X_r-X_{t_i}|^p {\rm d}r.
\end{align*} 
The above two inequalities, in combination with Gr\"onwall inequality, complete the proof of Theorem \ref{main1}.

\subsection{Proof of Theorem \ref{main2}}

For $t \in [0, T]$, denote $e_t:=X_t-Y_t^M$.
It follows from Eq. \eqref{sde} and \eqref{int-ito} that 
\begin{align*}
{\rm d}e_t & = (b(X_t)-b(Y^M_{\eta^+_M(t)}) {\rm d}t + (\sigma(X_t)-\sigma(Y^M_{\eta_M(t)}) {\rm d}W_t. 
\end{align*} 

Applying It\^o formula, we have
\begin{align*} 
|e_t|^p
& =p \int_0^t |e|^{p-2} \<e, b(X)-b(Y^M)\> {\rm d} r 
+ p \int_0^t |e|^{p-2} \<e, b(Y^M)-b(Y^M_{\eta^+_M})\> {\rm d} r \\
& \quad + \frac12 p(p-1) \int_0^t |e|^{p-2} |\sigma(X)- \sigma(Y^M)
+ \sigma(Y^M)-\sigma(Y^M_{\eta_M})|_{HS}^2 {\rm d}r \\
& \quad + p \int_0^t |e|^{p-2} \<e, \sigma(X)-\sigma(Y^M)
+\sigma(Y^M)-\sigma(Y^M_{\eta_M}) {\rm d}W\>  \\
& \le p \int_0^t |e|^{p-2} [\<e, b(X)-b(Y^M)\>+(p-1) |\sigma(X)- \sigma(Y^M)|_{HS}^2]{\rm d} r \\ 
& \quad + p \int_0^t |e|^{p-2} [\<e, b(Y^M)-b(Y^M_{\eta^+_M})\>+ (p-1) |\sigma(Y^M)-\sigma(Y^M_{\eta_M})|_{HS}^2] {\rm d}r \\  
& \quad + p \int_0^t |e|^{p-2} \<e, \sigma(X)-\sigma(Y^M) 
+\sigma(Y^M)-\sigma(Y^M_{\eta_M}) {\rm d}W\>.
\end{align*} 
By (A1), 
\begin{align*} 
& p \int_0^t |e|^{p-2} [\<e, b(X)-b(Y^M)\>+(p-1) |\sigma(X)- \sigma(Y^M)|_{HS}^2]{\rm d} r \\
& \le p(L_b+(p-1) L_\sigma^2) \int_0^t |e|^p {\rm d} r. 
\end{align*} 
By Young inequality, for any $\epsilon>0$, there exists $C_\epsilon$ such that 
\begin{align*}  
& p \int_0^t |e|^{p-2} [\<e, b(Y^M)-b(Y^M_{\eta^+_M})\>+ (p-1) |\sigma(Y^M)-\sigma(Y^M_{\eta_M})|_{HS}^2] {\rm d}r \\
& \le p \epsilon \int_0^t |e|^p {\rm d} r
+ C_\epsilon \int_0^t |b(Y^M)-b(Y^M_{\eta^+_M})|^p
+|Y^M-Y^M_{\eta_M}|^p {\rm d} r.
\end{align*} 
Therefore,
\begin{align*} 
\ee \sup_{t \in [0, T]} |e_t|^p
& \le p(L_b+(p-1) L_\sigma^2+\epsilon) \int_0^t \ee|e|^p {\rm d} r \\ 
& \quad + C_\epsilon \int_0^t \ee|b(Y^M)-b(Y^M_{\eta^+_M})|^p
+ \ee |Y^M-Y^M_{\eta_M}|^p {\rm d} r \\  
& \quad + \ee \sup_{t \in [0, T]} \Big| p \int_0^t |e|^{p-2} \<e, \sigma(X)-\sigma(Y^M)+\sigma(Y^M)-\sigma(Y^M_{\eta_M}) {\rm d}W\> \Big|.
\end{align*}

By Burkholder--Davis--Gundy and Young inequalities, 
\begin{align*} 
& \ee \sup_{t \in [0, T]} \Big| p \int_0^t |e|^{p-2} \<e, \sigma(X)-\sigma(Y^M)+\sigma(Y^M)-\sigma(Y^M_{\eta_M}) {\rm d}W\> \Big|  \\
& \le 3p \ee \Big(\int_0^T |\sigma(X)-\sigma(Y^M)+\sigma(Y^M)-\sigma(Y^M_{\eta_M})|_{HS}^2 |e|^{2(p-1)} {\rm d}r\Big)^\frac12 \\
& \le 3p \ee \Big[ \sup_{t \in [0,T]}|e_t|^{p-1}
\Big( \int_0^T |\sigma(X)-\sigma(Y^M)+\sigma(Y^M)-\sigma(Y^M_{\eta_M})|_{HS}^2  {\rm d}r\Big)^\frac12 \Big]\\
& \le \frac12 \ee \sup_{t \in [0,T]}|e_t|^p
+ 3p C_p \ee \Big( \int_0^T |\sigma(X)-\sigma(Y^M)+\sigma(Y^M)-\sigma(Y^M_{\eta_M})|_{HS}^2  {\rm d}r\Big)^\frac p2 \\
& \le \frac12 \ee \sup_{t \in [0,T]}|e_t|^p
+ 3p C_p 2^{p-1} \ee \Big( \int_0^T |\sigma(X)-\sigma(Y^M)|_{HS}^2  {\rm d}r\Big)^\frac p2 \\
& \quad + 3p C_p 2^{p-1} \ee \Big( \int_0^T |\sigma(Y^M)-\sigma(Y^M_{\eta_M})|_{HS}^2  {\rm d}r\Big)^\frac p2.
\end{align*} 
It follows that 
\begin{align*} 
& \ee |e_t|^p
\le C \int_0^t \ee|e|^p {\rm d} r
+ C \int_0^t \ee|b(Y^M)-b(Y^M_{\eta^+_M})|^p
+ \ee |Y^M-Y^M_{\eta_M}|^p {\rm d} r, 
\end{align*} 
from which we conclude Theorem \ref{main2} by Gr\"onwall inequality.

\subsection{Proof of Theorem \ref{main3}}

We begin with the following H\"older continuity of the exact solution in the whole time interval and moment's increment estimate of the interpolation process in subintervals.

\begin{lm}\label{lm-hol-x}
Let Assumptions \ref{ap-mon} and \ref{ap-b} hold for some $L_b \in \rr$, $\widetilde{L_b}, L_\sigma \ge 0$, and $q \ge 1$.
Assume that $\ee |X_0|^{p q}<\infty$ for some $p \ge 2$, then there exists s constant $C$ such that 
\begin{align} \label{hol-x}
\ee |X_t-X_s|^p 
\le C  (1+ \ee |X_0|^{p q}) |t-s|^{p/2},
\quad t, s \in [0, T].
\end{align} 
\end{lm}

\begin{proof} 
Without loss of generality, we assume $0 \le s \le t \le T$.
Then  
\begin{align*}
X_t-X_s =\int_s^t b(X) {\rm d}r + \int_s^t \sigma(X) {\rm d}W. 
\end{align*}  
Using Burkholder--Davis--Gundy and H\"older inequalities and the conditions \eqref{coe}-\eqref{gro-g} and \eqref{A2}, we get  
\begin{align*}
\ee |X_t-X_s|^p 
& \le C (t-s)^{p-1} \int_s^t \ee |b(X)|^p {\rm d}r  
+ C (t-s)^{\frac p2-1} \int_s^t \ee|\sigma(X)|_{HS}^p {\rm d}W  \\
& \le C (t-s)^p \sup_{t \in [0, T]}\ee |b(X_t)|^p  
+ C (t-s)^\frac p2 \sup_{t \in [0, T]} \ee|\sigma(X_t)|_{HS}^p  \\
& \le C (t-s)^\frac p2 (1+\sup_{t \in [0, T]} \ee|X_t|^{p q}).
\end{align*}
Then we get \eqref{hol-x} by the above estimate and the estimate \eqref{sta-sde}.
\end{proof} 

By \eqref{be} and \eqref{int}, for any $t \in [0, T]$, 
\begin{align} 
Y^M_t-Y^M_{\eta_M(t)} 
& = b(Y^M_{\eta^+_M(t)})(t-\eta_M(t)) 
+ \sigma(Y^M_{\eta_M(t)}) (W_t-W_{\eta_M(t)}), \label{yt} \\
Y^M_t-Y^M_{\eta^+_M(t)} 
& =b(Y^M_{\eta^+_M(t)})(\eta^+_M(t)-t) 
+ \sigma(Y^M_{\eta_M(t)}) (W_{\eta^+_M(t)}-W_t). \label{yt+}
\end{align}

\begin{lm}\label{lm-hol-yt}
Let Assumptions \ref{ap-mon} and \ref{ap-b} hold for some $L_b \in \rr$, $\widetilde{L_b}, L_\sigma \ge 0$, and $q \ge 1$.
Assume that $\ee |X_0|^{p q}<\infty$ for some $p \ge 4$, then for any $M \in \nn_+$ such that $L_b \tau<1$, there exists $C$ such that for any $t \in [0, T]$, 
\begin{align} 
\ee |Y^M_t-Y^M_{\eta_M(t)}|^p 
& \le C  (1+ \ee |X_0|^{p q}) \tau^{p/2}, \label{err-yt} \\
\ee |Y^M_t-Y^M_{\eta^+_M(t)}|^p 
& \le C  (1+ \ee |X_0|^{p q}) \tau^{p/2}. \label{err-yt+} 
\end{align}
\end{lm}

\begin{proof} 
By \eqref{yt} and the conditions \eqref{A1} and \eqref{A2},  
\begin{align*} 
& \ee |Y^M_s-Y^M_{\eta_M(s)}|^p 
= \ee |b(Y^M_{\eta^+_M(s)})(s-\eta_M(s)) 
+ \sigma(Y^M_{\eta_M(s)}) (W_s-W_{\eta_M(s)})|^p \\
& \le C_p (1+ \ee |Y^M_{\eta_M(s)}|^p
+ \ee |Y^M_{\eta^+_M(s)}|^{p q}) (|s-\eta_M(s)|^p
+ \ee|W_s-W_{\eta_M(s)}|^p) \\
& \le C_p (1+ \sup_{k \in \zz_M} \ee |Y^M_k|^{p q}) \tau^{p/2},
\end{align*}  
which, in combination with the moment's estimation \eqref{sta-yk}, shows \eqref{err-yt}.
Similar arguments, together with \eqref{yt+}, imply \eqref{err-yt+}. 
\end{proof}

Now we can give the proof of Theorem \ref{main3}.
By \eqref{main10}, the condition \eqref{A2}, H\"older inequality, and the H\"older estimate \eqref{hol-x}, we have
\begin{align*}   
& \sup_{k \in \zz_M} \ee |X_{t_k}-Y_k^M|^p \\
& \le C \sum_{i=0}^{M-1} \int_{t_i}^{t_{i+1}} \ee |b(X_r)-b(X_{t_{i+1}})|^p + \ee | X_r-X_{t_i}|^p {\rm d}r \\
& \le C (1+ \sup_{t \in [0, T]} \ee |X_t|^{p(2q-1)})^\frac{q-1}{2q-1}
\Big(\sup_{t \neq s} \frac{\ee |X_t-X_s|^{\frac{p (2q-1)}q}}{|t-s|^{\frac{p (2q-1)}{2q}}} \Big)^\frac q{2q-1} \tau^\frac p2
\\ 
& \le C (1+\ee|X_0|^{p (2q-1)}) \tau^\frac p2.
\end{align*}  
This shows \eqref{err-sde}.
Similarly,
\begin{align*}  
& \ee \sup_{t \in [0, T]} |X_t-Y_t^M|^p \\
& \le C (1+\sup_{t \in [0, T]} \ee|Y^M_t|^{p (2q-1)})^\frac{q-1}{2q-1} \sup_{t \in [0, T]} (\ee |Y^M_t-Y^M_{\eta^+_M(t)}|^{\frac{p (2q-1)}q})^\frac q{2q-1} \\ 
& \le C (1+\ee|X_0|^{pq (2q-1)}) \tau^\frac p2,
\end{align*} 
which shows \eqref{err-sde+} and we complete the proof of Theorem \ref{main3}.

\section{$L^p$-Convergence Rate of Galerkin-Based BE Schemes of Monotone SPDEs}
\label{sec4}

In the last section, we apply our main results, Theorems \ref{main1}-\ref{main3}, to the following second-order SPDEs with polynomial growth coefficients:
\begin{align} \label{spde}
\begin{split}
& \frac{\partial X_t(\xi)}{\partial t}  
=\Delta X_t(\xi) + b(X_t(\xi)) 
+\sigma(X_t(\xi)) \frac{{\rm d} W_t(\xi)}{{\rm d}  t}, \quad (t,\xi)\in (0, T] \times \OOO, \\
& X_t(\xi)=0,\quad (t,\xi)\in [0, T] \times \partial \OOO,
\end{split}
\end{align} 
where $\OOO \subset \rr^d$ with $d=1,2,3$ is an open, bounded set with piecewise smooth boundary.

For $r \in [2, \infty]$ and $\theta \in \rr$, we set $(L^r=L^r(\OOO), \|\cdot\|_{L^r})$ and $(\dot H^{-\theta}=\dot H^{-\theta}(\OOO), \|\cdot\|_\theta)$ the usual Lebesgue and Sobolev interpolation spaces, respectively.
In particular, we denote $H:=L^2$, $V=\dot H^1$, and $V^*=\dot H^{-1}$.
Denote the inner product and norm in $H$ by $\|\cdot\|$ and $\<\cdot, \cdot\>$, respectively.
The norms in $V$ and $V^*$ and the dual between them are denoted by $\|\cdot\|_1$, $\|\cdot\|_{-1}$,  and $_1\<\cdot, \cdot\>_{-1}$, respectively.
Denote by $H_0:=Q^\frac12 H$ and 
$(\LL_2^0:=HS(H_0; H), \|\cdot\|_{\LL_2^0})$, $(\LL_2^1:=HS(H_0; V), \|\cdot\|_{\LL_2^1})$ the space of Hilbert--Schmidt operators from $H_0$ to $H$ and $V$, respectively. 
We also use $\LL(H_0, H)$ to denote the space of bounded linear operators $H_0$ to $H$.

The driving process $W$ in Eq. \eqref{spde} is an infinite-dimensional $H$-valued Wiener process, which has the Karhunen--Lo\`eve expansion 
\begin{align*} 
W(t)  =\sum_{k\in \nn_+} \sqrt{\lambda_k} g_k \beta_k(t),
\quad t\in [0,T].
\end{align*}
Here $\{g_k\}_{k=1}^\infty$ forming an orthonormal basis of $H$ are the eigenvectors of $Q$ subject to the eigenvalues $\{\lambda_k\}_{k=1}^\infty$, and $\{\beta_k\}_{k=1}^\infty $ are mutually independent Brownian motions in $(\Omega, \FFF, \ff, \pp)$.
We mainly focus on trace-class noise, i.e., $Q$ is a trace-class operator, or equivalently, $\sum_{k \in \nn_+} \lambda_k<\infty$.

For $q \ge 1$, define by $F: L^{(q+1)'} \rightarrow L^{q+1}$ the Nemytskii operator associated with $b$:
\begin{align} \label{F}
F(u)(\xi):=b(u(\xi)),\quad u\in V,\ \xi \in \OOO.
\end{align} 
where $(q+1)'$ denote the conjugation of $q+1$, i.e., $1/(q+1)'+1/(q+1)=1$.
Then it follows from Assumption \ref{ap-b} that the operator $F$ defined in \eqref{F} has a continuous extension from $L^{q+1}$ to $L^{(q+1)'}$ and satisfies
\begin{align} \label{con-F}
_{L^{(q+1)'}}\<F(x)-F(y), x-y\>_{L^{q+1}} \le L_b \|x-y\|^2,
\quad x,y\in L^q,
\end{align}
where $_{L^{(q+1)'}} \<\cdot, \cdot\>_{L^{q+1}}$ denotes the dual between $L^{(q+1)'}$ and $L^{q+1}$.
In particular, if $q \ge 1$ for $d=1,2$ and $q\in [1,3]$ for $d=3$, one has the embeddings
\begin{align} \label{emb} 
V \subset L^{2q} \subset  L^{q+1} \subset H \subset L^{(q+1)'}\subset L^{(2q)'}\subset V^*,
\end{align} 
so that \eqref{con-F} leads to 
\begin{align} \label{F-mon} 
_1\<u-v, F(u)-F(v)\>_{-1}  \le L_b \|u-v\|^2,
\quad u,v \in V.
\end{align}
Denote by $G: H\rightarrow \LL(H_0, H)$ the Nemytskii operator associated with $\sigma$:
\begin{align} \label{G}
G(u) g_k(\xi):=\sigma(u(\xi)) g_k(\xi),\quad u\in H,\ k \in \nn_+,\ \xi\in \OOO.
\end{align}
Then the SPDE \eqref{spde} is equivalent to the stochastic evolution equation 
\begin{align} \label{see}
{\rm d}X_t =(AX_t+F(X_t)) {\rm d}t+G(X_t) {\rm d}W_t, \quad t \in (0, T].
\end{align}   

We focus on Galerkin-based BE schemes of Eq. \eqref{see}.
Let $h\in (0,1)$, $\TT_h$ be a regular family of partitions of $\OOO$ with maximal length $h$, and $V_h \subset V$ be the space of continuous functions on $\bar \OOO$ which are piecewise linear over $\TT_h$ and vanish on the boundary $\partial \OOO$.
Let $A_h: V_h \rightarrow V_h$ and $\PP_h: V^* \rightarrow V_h$ be the discrete Laplacian and generalized orthogonal projection operators, respectively, defined by 
\begin{align*}  
\<A_h u^h, v^h\> & =-\<\nabla u^h, \nabla v^h\>,
\quad u^h, v^h\in V_h,  \\
\<\PP_h u, v^h\> & =_1\<v^h, u\>_{-1},
\quad u\in V^*,\ v^h\in V_h. 
\end{align*}
Let $N \in \nn_+$ and $V_N$ be the space spanned by the first $N$-eigenvectors of the Dirichlet Laplacian operator which vanish on the boundary $\partial \OOO$.
Similarly, one can define spectral Galerkin approximate Laplacian and generalized orthogonal projection operators, respectively, as 
\begin{align*}  
\<A_N u^N, v^N\> & =-\<\nabla u^N, \nabla v^N\>,
\quad u^N, v^N \in V_N,  \\
\<\PP_N u, v^N\> & =_1\<v^N, u\>_{-1},
\quad u\in V^*,\ v^N \in V_N. 
\end{align*}

Our main condition on the diffusion operator $G$ defined in \eqref{G} is the following assumption.

\begin{ap} \label{ap-g}
The operator $G:H\rightarrow \LL_2^0$ defined in \eqref{G} is Lipschitz continuous, i.e., there exists a constant $L_\sigma \ge 0$ such that 
\begin{align}  \label{B2} \tag{B2} 
\|G(u)-G(v)\|_{\LL_2^0} \le L_\sigma \|u-v\|,
\quad u,v\in H.
\end{align}
Moreover, $G(\dot H^1)\subset \LL_2^1$ and there exists a constant $\widetilde{L_\sigma} \ge 0$ such that 
\begin{align} \label{B3}\tag{B3} 
\|G(z) \|_{\LL_2^1} \le \widetilde{L_\sigma}(1+ \|z\|_1),
\quad z \in V.  
\end{align}
\end{ap}

Under conditions \eqref{A1}, \eqref{A2}, \eqref{B2},  and \eqref{B3}, Eq. \eqref{spde}, or equivalently, Eq. \eqref{see} with initial datum $X_0 \in V$ possesses a unique $\ff$-adapted solution $\{X_t\}_{t \in [0, T]}$ in $V$ with continuous sample paths, see \cite{Liu13(JDE)}.
Moreover, we proved in \cite[Theorem 3.1]{LQ20(SPDE)} the following moment's estimation provided $X_0\in L^p(\Omega; V)$ with $p \ge q+1$, $q \ge 1$ for $d=1,2$ and $q\in [1,3]$ for $d=3$:
\begin{align} \label{sta-spde}
\ee \sup_{t\in [0,T]} \|X_t\|^p_1 
 \le C(1+ \ee \|X_0\|^p_1 ).
\end{align} 
  
The Galerkin finite element-based BE scheme for Eq. \eqref{see} is to find a sequence of $\ff$-adapted $V_h$-valued process $\{X_m^h:\ m \in \zz_M\}$ such that
\begin{align}\label{be1} 
X^h_{m+1}
=X_m^h+\tau A_h X^h_{m+1}
+\tau \PP_h F(X^h_{m+1})
+\PP_h G(X_m^h) \delta_m W,  \quad m \in \zz_{M-1},
\end{align}
with the initial value $X^h_0=\PP_h X_0$.
We call it finite element BE scheme.  
It is clear that the finite element BE scheme \eqref{be1} is equivalent to the scheme
\begin{align}\label{full}
X^h_{m+1}=S_{h,\tau} X_m^h+\tau S_{h,\tau} \PP_h F(X^h_{m+1})
+S_{h,\tau} \PP_h G(X_m^h) \delta_m W,
\quad m\in \zz_{M-1},
\end{align}
with initial datum $X^h_0=\PP_h H_0$, where $S_{h,\tau}:=({\rm Id}-\tau A_h)^{-1}$ is a space-time approximation of the continuous semigroup $S$ in one step.
Iterating \eqref{full} for $m$-times, we obtain for $m\in \zz_{M-1}$,
\begin{align}\label{full-sum}
X^h_{m+1}
=S_{h,\tau}^{m+1} X^h_0+\tau \sum_{i=0}^m S_{h,\tau}^{m+1-i} \PP_h F(X_h^{i+1})
+\sum_{i=0}^m S_{h,\tau}^{m+1-i} \PP_h G(X_h^i) \delta_i W.
\end{align}

Similarly, the spectral Galerkin BE scheme for Eq. \eqref{see} is to find a sequence of $\ff$-adapted $V_h$-valued process $\{X_m^N:\ m \in \zz_M\}$ such that
\begin{align}\label{be2} 
X^N_{m+1}
=X_m^N+\tau A_N X^N_{m+1}
+\tau \PP_N F(X^N_{m+1})
+\PP_N G(X_m^N) \delta_m W,  \quad m \in \zz_{M-1},
\end{align}
with the initial value $X^N_0=\PP_N X_0$. 

Using the idea in the SODEs setting in Section \ref{sec3}, we have the following $L^p$-convergence rate of the finite element BE scheme \eqref{be1}.
Note that the case $p=2$ had been shown in \cite{LQ20(SPDE)}.

\begin{tm}  \label{tm-err-spde}
Let \eqref{A1}, \eqref{A2}, \eqref{B2}, and \eqref{B3} hold for some $L_b \in \rr$, $\widetilde{L_b}, L_\sigma, \widetilde{L_\sigma} \ge 0$ and $q \ge 1$ for $d=1,2$ and $q\in [1,3]$ for $d=3$.
Assume that $\ee \|X_0\|_1^{p (q^2+q-1)}<\infty$ for some $p \ge \max\{4, q+1\}$, then for any $M \in \nn_+$ such that $2 L_b \tau<1$, there exists $C$ such that     
\begin{align} 
& \sup_{k \in \zz_M} \ee \|X_{t_k}-X^h_k\|^p 
\le C (1+\ee|X_0|^{p (q^2+q-1)}) (h^p+\tau^{p/2}), \label{err-spde} \\
& \sup_{k \in \zz_M} \ee \|X_{t_k}-X^N_k\|^p 
\le C (1+\ee|X_0|^{p (q^2+q-1)}) (N^{-p}+\tau^{p/2}). \label{err-spde+}
\end{align}  
\end{tm}

\begin{proof}
It suffices to prove \eqref{err-spde}; similar arguments  would immediately yield \eqref{err-spde+}.
To this end, we introduce the auxiliary process 
\begin{align}\label{aux1}
\widetilde{X}^h_{m+1}
=S_{h,\tau}^{m+1} X^h_0
+\tau \sum_{i=0}^m S_{h,\tau}^{m+1-i} \PP_h F(X_{t_{i+1}})
+\sum_{i=0}^m S_{h,\tau}^{m+1-i} \PP_h G(X_{t_i}) \delta_i W,
\end{align}
for $m\in \zz_{M-1}$, where the terms $X_h^{i+1}$ and $X_h^i$ in the discrete deterministic and stochastic convolutions of \eqref{full-sum} are replaced by $X_{t_{i+1}}$ and $X_{t_i}$, respectively.
From \eqref{aux1}, it is clear that 
\begin{align}\label{aux1+}
\widetilde{X}^h_{m+1}
=\widetilde{X}_m^h+\tau A_h \widetilde{X}^h_{m+1}
+\tau \PP_h F(X_{t_{m+1}})
+\PP_h G(X_{t_m}) \delta_m W.
\end{align}
Using the argument in Proposition \ref{sta}, it is not difficult to show that 
\begin{align}\label{sta-aux1}
\sup_{k \in \zz_M} \ee \|\widetilde{X}^h_k\|_1^p
\le C (1+ \ee\|X_0\|_1^{p q}).
\end{align}

In \cite[Lemma 4.2]{LQ20(SPDE)}, we had proved that  
\begin{align} \label{err-spde0}
\sup_{k \in \zz_M} \ee \|X_{t_k}-\widetilde{X}^h_k\|^p 
\le C (1+ \ee \|X_0\|_1^{p(2q-1)}) (h^p+\tau^\frac p2).
\end{align} 
It remains to estimate $\sup_{k \in \zz_M} \ee \|\widetilde{X}^h_k-X^h_k\|^p$.
For $k \in \zz_M$, denote  
$e^h_k: =\widetilde{X}^h_k-X^h_k$.
Then by \eqref{be1} and \eqref{aux1+} we have 
\begin{align} \label{eq-err-spde}
e^h_{k+1}-e^h_k 
& = \tau \Delta e^h_{k+1}
+ \tau (F(X_{t_{k+1}})-F(X^h_{k+1}) ) 
+ (G(X_{t_k})-G(X^h_k)) \delta_k W.
\end{align} 
Testing $e^h_{k+1}$ and using the inequality \eqref{ab} which holds also for all $\alpha, \beta \in H$ and the condition \eqref{F-mon}, we obtain
\begin{align*} 
& \frac12 ( \|e^h_{k+1}\|^2 - \|e^h_k\|^2) 
+  \frac12 \|e^h_{k+1}-e^h_k\|^2 
+ \|\nabla e^h_{k+1}\|^2 \tau \\
&= _1\<e^h_{k+1}, F(X_{t_{k+1}})-F(\widetilde{X}^h_{k+1})\>_{-1} \tau
+ _1\<e^h_{k+1}, F(\widetilde{X}^h_{k+1})-F(X^h_{k+1}) \>_{-1} \tau \\
& \quad + \<e^h_{k+1}-e^h_k, (G(X_{t_k})-G(X^h_k)) \delta_k W\>
+ \<e^h_k, (G(X_{t_k})-G(X^h_k)) \delta_k W\>\\
&\le \|\nabla e^h_{k+1}\|^2 \tau
+L_b \tau \|e^h_{k+1}\|^2 + \<e^h_k, (G(X_{t_k})-G(X^h_k)) \delta_k W\> \\
&\quad +C \tau \|F(X_{t_{k+1}})-F(\widetilde{X}^h_{k+1})\|^2_{-1}  
\nonumber \\
&\quad + \frac12 \|e^h_{k+1}-e^h_k\|^2
+ L_\sigma^2 \|X_{t_k}-\widetilde{X}^h_k\|^2 \|\delta_k W\|^2
+L_\sigma^2 \|e^h_k\|^2 \|\delta_k W\|^2.
\end{align*} 
It follows that 
\begin{align*}  
(1-2 L_b \tau) \|e^h_{k+1}\|^2 
& \le (1+2 L_\sigma^2 \|\delta_k W\|^2) \|e^h_k\|^2 
+C \tau \|F(X_{t_{k+1}})-F(\widetilde{X}^h_{k+1})\|^2_{-1}   \\
&\quad + 2 L_\sigma^2 \|X_{t_k}-\widetilde{Y}^M_k\|^2 \|\delta_k W\|^2 
+ 2 \<e^h_k, (G(X_{t_k})-G(X^h_k)) \delta_k W\>. 
\end{align*}

Using the same arguments to derive \eqref{err-ek} in the proof of Theorem \ref{main1}, we get 
\begin{align*}
\sup_{k \in \zz_M} \ee \|e^h_k\|^p  
& \le C \tau \sum_{i=0}^{M-1} \ee \|e^h_i\|^p
+ C \tau^\frac p2 \ee \Big( \sum_{i=0}^{M-1} \|F(X_{t_{i+1}})-F(\widetilde{X}^h_{i+1})\|^2_{-1}\Big)^\frac p2 \nonumber \\
& \quad + C \ee \Big( \sum_{i=0}^{M-1} \|X_{t_i}-\widetilde{Y}^M_i\|^2 \|\delta_k W\|^2  \Big)^\frac p2  \nonumber \\
& \quad + C \ee \Big| \sum_{i=0}^{M-1} \<e^h_i, (G(X_{t_i})-G(X^h_i)) \delta_i W\> \Big|^\frac p2 \nonumber  \\
& \le C \tau \sum_{i=0}^{M-1} \ee \|e^h_i\|^p
+ C \sup_{k \in \zz_M} \ee \|F(X_{t_k})-F(\widetilde{X}^h_k)\|^p_{-1} \nonumber \\
& \quad + C \sup_{k \in \zz_M} \ee \|X_{t_k}-\widetilde{X}^h_k\|^p
+ C  \ee \Big| \sum_{i=0}^{M-1} \<e^h_i, (G(X_{t_i})-G(X^h_i)) \delta_i W\> \Big|^\frac p2. 
\end{align*} 
Applying discrete Burkholder inequality as used in \eqref{dis-BDG} and Cauchy--Schwarz inequality, we have
\begin{align*}   
\ee \Big| \sum_{i=0}^{M-1} \<e^h_i, (G(X_{t_i})-G(X^h_i)) \delta_i W\> \Big|^\frac p2 
\le C \tau \sum_{i=0}^{k-1} \ee \|e^h_i\|^p 
+ C \sup_{k \in \zz_M} \ee \|X_{t_k}-\widetilde{X}^h_k\|^p.
\end{align*}  
Consequently, we get 
\begin{align} \label{ekh}
\sup_{k \in \zz_M} \ee \|e^h_k\|^p  
\le C \big(\sup_{k \in \zz_M} \ee \|X_{t_k}-\widetilde{X}^h_k\|^p
+\sup_{k \in \zz_M} \ee \|F(X_{t_k})-F(\widetilde{X}^h_k)\|^p_{-1} \big). 
\end{align} 

It follows from the embedding \eqref{emb} that 
\begin{align} \label{f-1}
\|F(u)-F(v)\|_{-1} 
\le C (1+\|u\|^{q-1}_{L^{2q}}+\|v\|^{q-1}_{L^{2q}} ) \|u-v\|,
\quad u,v \in V.
\end{align} 
Then \eqref{ekh} implies 
\begin{align*}
\sup_{k \in \zz_M} \ee \|e^h_k\|^p  
& \le C \sup_{k \in \zz_M} \ee [(1+\|X_{t_k}\|^{p(q-1)}_1+\|\widetilde{X}^h_k\|^{p(q-1)}_1 ) \|X_{t_k}-\widetilde{X}^h_k\|^p] \\
& \le C (1+\sup_{t \in [0, T]} \ee \|X_t\|_1^{\frac{p (q^2+q-1)}q}
+\sup_{k \in \zz_M} \ee \|\widetilde{X}^h_k\|_1^{\frac{p (q^2+q-1)}q})^\frac{q(q-1)}{q^2+q-1} \\ 
& \qquad \times \sup_{t \in [0, T]} (\ee \|X_{t_k}-\widetilde{X}^h_k\|^{\frac{p (q^2+q-1)}{2q-1}})^\frac{2q-1}{q^2+q-1}.
\end{align*} 
Then we conclude \eqref{err-spde} by the above estimate, the stability \eqref{sta-aux1}, and the error estimate \eqref{err-spde0}.
\end{proof}

\bibliographystyle{amsalpha}
\bibliography{bib}

\newcommand{\etalchar}[1]{$^{#1}$}
\providecommand{\bysame}{\leavevmode\hbox to3em{\hrulefill}\thinspace}
\providecommand{\MR}{\relax\ifhmode\unskip\space\fi MR }
\providecommand{\MRhref}[2]{%
  \href{http://www.ams.org/mathscinet-getitem?mr=#1}{#2}
}
\providecommand{\href}[2]{#2}
\begin{thebibliography}{WWD20}

\bibitem[AK17]{AK17(BIT)}
A.~Andersson and R.~Kruse, \emph{Mean-square convergence of the
  {BDF}2-{M}aruyama and backward {E}uler schemes for {SDE} satisfying a global
  monotonicity condition}, BIT \textbf{57} (2017), no.~1, 21--53. \MR{3608309}

\bibitem[BCH19]{BCH19(IMA)}
C.-E. Br\'{e}hier, J.~Cui, and J.~Hong, \emph{Strong convergence rates of
  semidiscrete splitting approximations for the stochastic {A}llen-{C}ahn
  equation}, IMA J. Numer. Anal. \textbf{39} (2019), no.~4, 2096--2134.
  \MR{4019051}

\bibitem[BGJK]{BGJK17}
S.~Becker, B.~Gess, A.~Jentzen, and P.~Kloeden, \emph{Strong convergence rates
  for explicit space-time discrete numerical approximations of stochastic
  {A}llen--{C}ahn equations}, arXiv:1711.02423.

\bibitem[BHJ{\etalchar{+}}]{BHJKLS19}
M.~Beccari, M.~Hutzenthaler, A.~Jentzen, R.~Kurniawan, F.~Lindner, and
  D.~Salimova, \emph{Strong and weak divergence of exponential and
  linear-implicit {E}uler approximations for stochastic partial differential
  equations with superlinearly growing nonlinearities}, arXiv:1903.06066.

\bibitem[BJ19]{BJ19(SPA)}
S.~Becker and A.~Jentzen, \emph{Strong convergence rates for
  nonlinearity-truncated {E}uler-type approximations of stochastic
  {G}inzburg-{L}andau equations}, Stochastic Process. Appl. \textbf{129}
  (2019), no.~1, 28--69. \MR{3906990}

\bibitem[CHL17]{CHL17(JDE)}
J.~Cui, J.~Hong, and Z.~Liu, \emph{Strong convergence rate of finite difference
  approximations for stochastic cubic {S}chr\"{o}dinger equations}, J.
  Differential Equations \textbf{263} (2017), no.~7, 3687--3713. \MR{3670034}

\bibitem[CHLZ19]{CHLZ19(JDE)}
J.~Cui, J.~Hong, Z.~Liu, and W.~Zhou, \emph{Strong convergence rate of
  splitting schemes for stochastic nonlinear {S}chr\"{o}dinger equations}, J.
  Differential Equations \textbf{266} (2019), no.~9, 5625--5663. \MR{3912762}

\bibitem[CJM16]{CJM16(SIFM)}
J.-F. Chassagneux, A.~Jacquier, and I.~Mihaylov, \emph{An explicit {E}uler
  scheme with strong rate of convergence for financial {SDE}s with
  non-{L}ipschitz coefficients}, SIAM J. Financial Math. \textbf{7} (2016),
  no.~1, 993--1021. \MR{3582409}

\bibitem[D{\"o}r12]{Dor12(SINUM)}
P.~D{\"o}rsek, \emph{Semigroup splitting and cubature approximations for the
  stochastic {N}avier-{S}tokes equations}, SIAM J. Numer. Anal. \textbf{50}
  (2012), no.~2, 729--746. \MR{2914284}

\bibitem[FG20]{FG20(AAP)}
W.~Fang and M.~B. Giles, \emph{Adaptive {E}uler-{M}aruyama method for {SDE}s
  with nonglobally {L}ipschitz drift}, Ann. Appl. Probab. \textbf{30} (2020),
  no.~2, 526--560. \MR{4108115}

\bibitem[FLZ17]{FLZ17(SINUM)}
X.~Feng, Y.~Li, and Y.~Zhang, \emph{Finite element methods for the stochastic
  {A}llen--{C}ahn equation with gradient-type multiplicative noises}, SIAM J.
  Numer. Anal. \textbf{55} (2017), no.~1, 194--216.

\bibitem[HJ15]{HJ15(MAMS)}
M.~Hutzenthaler and A.~Jentzen, \emph{Numerical approximations of stochastic
  differential equations with non-globally {L}ipschitz continuous
  coefficients}, Mem. Amer. Math. Soc. \textbf{236} (2015), no.~1112, v+99.
  \MR{3364862}

\bibitem[HJ20]{HJ20(AOP)}
\bysame, \emph{On a perturbation theory and on strong convergence rates for
  stochastic ordinary and partial differential equations with nonglobally
  monotone coefficients}, Ann. Probab. \textbf{48} (2020), no.~1, 53--93.
  \MR{4079431}

\bibitem[HJK11]{HJK11(PRSLSA)}
M.~Hutzenthaler, A.~Jentzen, and P.~E. Kloeden, \emph{Strong and weak
  divergence in finite time of {E}uler's method for stochastic differential
  equations with non-globally {L}ipschitz continuous coefficients}, Proc. R.
  Soc. Lond. Ser. A Math. Phys. Eng. Sci. \textbf{467} (2011), no.~2130,
  1563--1576. \MR{2795791}

\bibitem[HJK12]{HJK12(AAP)}
\bysame, \emph{Strong convergence of an explicit numerical method for {SDE}s
  with nonglobally {L}ipschitz continuous coefficients}, Ann. Appl. Probab.
  \textbf{22} (2012), no.~4, 1611--1641. \MR{2985171}

\bibitem[HMS02]{HMS02(SINUM)}
D.~Higham, X.~Mao, and A.~Stuart, \emph{Strong convergence of {E}uler-type
  methods for nonlinear stochastic differential equation}, SIAM J. Numer. Anal.
  \textbf{40} (2002), no.~3, 1041--1063. \MR{1949404}

\bibitem[JP20]{JP20(IMA)}
A.~Jentzen and P.~Pu\v{s}nik, \emph{Strong convergence rates for an explicit
  numerical approximation method for stochastic evolution equations with
  non-globally {L}ipschitz continuous nonlinearities}, IMA J. Numer. Anal.
  \textbf{40} (2020), no.~2, 1005--1050. \MR{4092277}

\bibitem[KL18]{KL18(IMA)}
C.~Kelly and G.~J. Lord, \emph{Adaptive time-stepping strategies for nonlinear
  stochastic systems}, IMA J. Numer. Anal. \textbf{38} (2018), no.~3,
  1523--1549. \MR{3829168}

\bibitem[KLL15]{KLL15(JAP)}
M.~Kov\'acs, S.~Larsson, and F.~Lindgren, \emph{On the backward euler
  approximation of the stochastic {A}llen--{C}ahn equation}, J. Appl. Prob.
  \textbf{52} (2015), 323--338.

\bibitem[KLL18]{KLL18(MN)}
M.~Kov\'{a}cs, S.~Larsson, and F.~Lindgren, \emph{On the discretisation in time
  of the stochastic {A}llen-{C}ahn equation}, Math. Nachr. \textbf{291} (2018),
  no.~5-6, 966--995. \MR{3795566}

\bibitem[KP92]{KP92}
P.~E. Kloeden and E.~Platen, \emph{Numerical solution of stochastic
  differential equations}, Applications of Mathematics (New York), vol.~23,
  Springer-Verlag, Berlin, 1992. \MR{1214374}

\bibitem[KR79]{KR79}
N.V. Krylov and B.L. Rozovskii, \emph{Stochastic evolution equations}, 71--147,
  256. \MR{570795}

\bibitem[Kru14]{Kru14}
R.~Kruse, \emph{Strong and weak approximation of semilinear stochastic
  evolution equations}, Lecture Notes in Mathematics, vol. 2093, Springer,
  Cham, 2014. \MR{3154916}

\bibitem[Liu13]{Liu13(JDE)}
W.~Liu, \emph{Well-posedness of stochastic partial differential equations with
  {L}yapunov condition}, J. Differential Equations \textbf{255} (2013), no.~3,
  572--592. \MR{3053478}

\bibitem[LMS07]{LMS07(IMA)}
H.~Lamba, J.~C. Mattingly, and A.~M. Stuart, \emph{An adaptive
  {E}uler-{M}aruyama scheme for {SDE}s: convergence and stability}, IMA J.
  Numer. Anal. \textbf{27} (2007), no.~3, 479--506. \MR{2337577}

\bibitem[LQ]{LQ20(SPDE)}
Z.~Liu and Z.~Qiao, \emph{Strong approximation of monotone stochastic partial
  differential equations driven by multiplicative noise}, Stoch. Partial
  Differ. Equ. Anal. Comput. (https://doi.org/10.1007/s40072-020-00179-2).

\bibitem[LQ20]{LQ20(IMA)}
\bysame, \emph{Strong approximation of monotone stochastic partial differential
  equations driven by white noise}, IMA J. Numer. Anal. \textbf{40} (2020),
  no.~2, 1074--1093. \MR{4092279}

\bibitem[LR15]{LR15}
W.~Liu and M.~R\"{o}ckner, \emph{Stochastic partial differential equations: an
  introduction}, Universitext, Springer, Cham, 2015. \MR{3410409}

\bibitem[QW19]{QW19(JSC)}
R.~Qi and X.~Wang, \emph{Optimal error estimates of {G}alerkin finite element
  methods for stochastic {A}llen-{C}ahn equation with additive noise}, J. Sci.
  Comput. \textbf{80} (2019), no.~2, 1171--1194. \MR{3977202}

\bibitem[QW20]{QW20(SINUM)}
\bysame, \emph{Error estimates of semidiscrete and fully discrete finite
  element methods for the {C}ahn-{H}illiard-{C}ook equation}, SIAM J. Numer.
  Anal. \textbf{58} (2020), no.~3, 1613--1653. \MR{4102717}

\bibitem[Sab16]{Sab16(AAP)}
S.~Sabanis, \emph{Euler approximations with varying coefficients: the case of
  superlinearly growing diffusion coefficients}, Ann. Appl. Probab. \textbf{26}
  (2016), no.~4, 2083--2105. \MR{3543890}

\bibitem[Wan20]{Wan20(SPA)}
X.~Wang, \emph{An efficient explicit full-discrete scheme for strong
  approximation of stochastic {A}llen-{C}ahn equation}, Stochastic Process.
  Appl. \textbf{130} (2020), no.~10, 6271--6299. \MR{4140034}

\bibitem[WWD20]{WWD20(BIT)}
X.~Wang, J.~Wu, and B.~Dong, \emph{Mean-square convergence rates of stochastic
  theta methods for {SDE}s under a coupled monotonicity condition}, BIT
  \textbf{60} (2020), no.~3, 759--790. \MR{4132905}

\end{thebibliography}

\end{document}